\documentclass[12pt, reqno]{amsart}

\usepackage{amsfonts,amssymb,latexsym,amsmath, amsxtra}
\usepackage{verbatim}

\usepackage[OT2,T1]{fontenc}
\DeclareSymbolFont{cyrletters}{OT2}{wncyr}{m}{n}
\DeclareMathSymbol{\Sha}{\mathalpha}{cyrletters}{"58}

\theoremstyle{plain}
\newtheorem{theorem}{Theorem}[section]
\newtheorem*{theorem*}{Theorem}
\newtheorem{corollary}[theorem]{Corollary}
\newtheorem{conjecture}[theorem]{Conjecture}
\newtheorem{lemma}[theorem]{Lemma}
\newtheorem{proposition}[theorem]{Proposition}
\newtheorem*{conjecture*}{Conjecture}

\theoremstyle{definition}

\theoremstyle{remark}

\newtheorem*{remark*}{Remark}
\newtheorem*{remarks*}{Remarks}

\numberwithin{equation}{section}

\newcommand{\R}{\mathbb R}
\newcommand{\N}{\mathbb N}
\newcommand{\Z}{\mathbb Z}
\newcommand{\C}{\mathbb C}

\makeatletter
\newcommand*{\rom}[1]{\expandafter\@slowromancap\romannumeral #1@}
\makeatother

\def\({\left(}
\def\){\right)}

\newcommand{\re}[1]{\text{Re}\(#1\)}

\begin{document}

\title
[On Dyson's crank conjecture]{On Dyson's crank conjecture and the uniform asymptotic behavior of certain inverse theta functions}
\author{Kathrin Bringmann}
\author{Jehanne Dousse}
\address{Mathematical Institute\\University of
Cologne\\ Weyertal 86-90 \\ 50931 Cologne \\Germany}
\email{kbringma@math.uni-koeln.de}
\address{LIAFA \\
Universite Denis Diderot - Paris 7 \\
75205 Paris Cedex 13 \\
FRANCE }
\email{jehanne.dousse@liafa.univ-paris-diderot.fr}

\subjclass[2010] {05A17, 11F03, 11F30, 11F50, 11P55, 11P82}

\date{\today}

\begin{abstract}
In this paper we prove a longstanding conjecture by Freeman Dyson concerning the limiting shape of the crank generating function. We fit this function in a more general family of inverse theta functions which play a key role in physics.

\end{abstract}

\maketitle

\section{Introduction and statement of results}
Dyson's crank was introduced to explain Ramanujan's famous partition congruences with modulus $5$, $7$, and $11$. Denoting for $n\in\N$ by $p(n)$ the number of integer partitions of $n$, Ramanujan \cite{Ra} proved that for $n\geq 0$
\begin{align*}
p\left( 5n+4\right) &\equiv 0 \pmod{5},\\
p\left( 7n+5\right) &\equiv 0 \pmod{7},\\
p\left( 11n+6\right) &\equiv 0 \pmod{11}.
\end{align*}
A key ingredient of his proof is the modularity of the partition generating function
$$
P(q):= \sum_{n= 0}^\infty p(n) q^n = \frac{1}{\left( q; q\right)_\infty} = \frac{q^{\frac{1}{24}}}{\eta (\tau)},
$$
where for $j\in\N_0 \cup \{\infty\}$ we set $(a)_j = (a ; q)_j := \prod_{\ell=0}^{j-1} (1-a q^\ell)$, $q:= e^{2\pi i \tau}$, and $\eta (\tau) := q^{\frac{1}{24}} \prod_{n=1}^\infty (1-q^n)$ is Dedekind's $\eta$-function, a modular form of weight $\frac12$.

Ramanujan's proof however gives little combinatorial insight into why the above congruences hold. In order to provide such an explanation, Dyson \cite{Dy1} famously introduced the {\it rank} of a partition, which is defined as its largest part minus the number of its parts. He conjectured that the partitions of $5n+4$ (resp. $7n+5$) form $5$ (resp. $7$) groups of equal size when sorted by their ranks modulo $5$ (resp. $7$). This conjecture was proven by Atkin and Swinnerton-Dyer \cite{AS}. Ono and the first author \cite{BO1} showed that partitions with given rank satisfy also Ramanujan-type congruences. Dyson further postulated the existence of another statistic which he called the ``crank'' and which should explain all Ramanujan congruences. The crank was later found by Andrews and Garvan \cite{AG, Ga}. If for a partition $\lambda$, $o(\lambda)$ denotes the number of ones in $\lambda$, and $\mu(\lambda)$ is the number of parts strictly larger than $o(\lambda)$, then the \emph{crank} of $\lambda$ is defined as
$$
\text{crank} (\lambda) := \left\{ \begin{array}{cc}  \text{largest part of $\lambda$}& \text{if $o(\lambda) =0$,} \\ \mu (\lambda) - o(\lambda) & \text{if $o(\lambda)>0$.} \end{array}\right.
$$
Denote by $M(m,n)$ the number of partitions of $n$ with crank $m$. Mahlburg \cite{Ma} then proved that partitions with fixed crank also satisfy Ramanujan-type congruences.
In this paper, we solve a longstanding conjecture by Dyson \cite{Dy2} concerning the limiting shape of the crank generating function.

\begin{conjecture} [Dyson]
As $n\to\infty$ we have
$$
M\left( m,n \right) \sim \frac14 \beta \mathrm{sech}^2 \left( \frac12 \beta m \right) p(n)
$$
with $\beta:= \frac{\pi}{\sqrt{6n}}$. 
\end{conjecture}
Dyson then asked the question about the precise range of $m$ in which this asymptotic holds and about the error term.
In this paper, we answer all of these questions.

\begin{theorem}\label{crank Theorem}
The Dyson-Conjecture is true. To be more precise, if $|m|\leq \frac1{\pi\sqrt{6}}\sqrt{n}\log n$, we have as $n\to\infty$
\begin{equation}\label{crankas1}
M(m, n)=\frac{\beta}{4}\mathrm{sech}^2\left(\frac{\beta m}{2}\right) p(n)\left(1+O\left(\beta^{\frac12}|m|^{\frac13}\right)\right).
\end{equation}
\end{theorem}

\begin{remarks*}\mbox{ }
\begin{enumerate}
\item[1.]
For fixed $m$ one can directly obtain asymptotic formulas since the generating function is the convolution of a modular form and a partial theta function \cite{BM}. However, Dyson's conjecture is a bivariate asymptotic. Indeed, this fact is the source of the difficulty of this problem.

\item[2.]
We note that Theorem  \ref{crank Theorem}  is of a very different nature than known asymptotics in the literature. For example, the partition function can be approximated as 
$$
p(n) = M(n) + O \left( n^{-\alpha}\right),
$$
where $M(n)$ is the main term which is  a sum of varying length of Kloosterman sums and Bessel functions and $\alpha>0$. Rademacher \cite{Rad} obtained $\alpha = \frac{3}{8}$, Lehmer improved this to $\frac12 - \varepsilon$ and Folsom and Masri \cite{FM} in their recent work obtained an impressive error of $n^{-\delta}$ for some absolute $\delta >\frac12$. Our result has a very different flavor due to the nonmodularity of the generating function and the bivariate asymptotics.

%
\item[3.]
In fact we could replace the error by $O( \beta^{\frac{1}{2}} m \alpha^2(m))$ for any $\alpha(m)$ such that $\frac{\log n}{n^{\frac{1}{4}}} = o \left(\alpha(m)\right)$ for all $|m|\leq \frac1{\pi\sqrt{6}}\sqrt{n}\log n$ and $\beta m \alpha(m) \rightarrow 0$ as $n \rightarrow \infty$. Here we chose $\alpha(m) = |m|^{-\frac{1}{3}}$ to avoid complicated expressions in the proof.

\end{enumerate}
\end{remarks*}
A straightforward calculation shows
\begin{corollary}\label{Almostall}
Almost all partitions satisfy Dyson's conjecture. To be more precise
\begin{equation}\label{crankas}
\sharp\left\{\lambda\vdash n|\text{\upshape crank}(\lambda)|\leq\frac{\sqrt{n}}{\pi\sqrt{6}}\log n\right\}\sim p(n).
\end{equation}
\end{corollary}

\begin{remarks*}\mbox{ }
\begin{enumerate}
\item[1.]
We thank Karl Mahlburg for pointing out Corollary \ref{Almostall} to us.
\item[2.]
We can improve \eqref{crankas} and give the size of the error term.
\end{enumerate}
\end{remarks*}
Dyson's conjecture follows from a more general result concerning the coefficients $M_k (m,n)$ defined for $k\in\N$ by
$$
\mathcal{C}_k \left( \zeta ; q \right) = \sum_{n=0}^\infty \sum_{m=-\infty}^\infty M_k \left( m,n \right) \zeta^m q^n := \frac{(q)_\infty^{2-k}}{\left( \zeta q\right)_\infty \left (\zeta^{-1} q\right)_\infty}.
$$
Note that $M(m, n)=M_1(m, n)$. Denoting by $p_k(n)$ the number of partitions of $n$ allowing $k$ colors, we have.

\begin{theorem}\label{1.4}For $k$ fixed and $|m|\leq\frac{1}{6\beta_k}\log n$, we have as $n\rightarrow\infty$
\[M_k(m,n)=\frac{\beta_k}{4}\operatorname{sech}^2\left(\frac{\beta_k m}{2}\right)p_k(n)\left(1+O\left(\beta_k^{\frac12} |m|^{\frac13}\right)\right),\]

with $\beta_k:=\pi\sqrt{\frac{k}{6n}}$.
\end{theorem}

\begin{remarks*}\mbox{ }
\begin{enumerate}

\item[1.]

We note that for $k\geq 3$, the functions $C_k(\zeta; q)$ are well-known to be generating functions of Betti 
numbers of moduli spaces of Hilbert schemes on $(k-3)-$point blow-ups of the projective plane 
\cite{Go} (see also \cite{BM} and references therein). The results of this paper immediately gives the 
limiting profile of the Betti numbers for large second Chern class of the sheaves. Recently,
Hausel and Rodriguez-Villegas  \cite{HRo} also determined  profiles of Betti numbers for other moduli spaces. 

\item[2.]
Note that our method of proof would allow determining further terms in the asymptotic expansion of $M_k(m, n)$.

\item[3.] Again we could replace the error by $O( \beta_k^{\frac{1}{2}} m \alpha_k^2(m) )$ for any $\alpha_k(m)$ such that $\frac{\log n}{(kn)^{\frac{1}{4}}} = o \left(\alpha_k(m)\right)$ for all $|m|\leq \frac1{6 \beta_k}\log n$ and $\beta_k m \alpha_k(m) \rightarrow 0$ as $n \rightarrow \infty$.
\item[4.]
The function $\mathcal{C}_k$ can also be represented as a so-called Lerch sum.  To be more precise, we have \cite{AG} 
\begin{equation}\label{CrankLerch}
\sum_{\substack{m\in\Z \\ n\geq 0}} M\left(m,n\right) \zeta^m q^n = \frac{1-\zeta}{(q)_\infty} \sum_{n\in\Z} \frac{(-1)^nq^{\frac{n(n+1)}2}}{1-\zeta q^n}.
\end{equation}
This representation, which was a key representation in \cite{BM}, is not used in this paper.
\item[5.]
The special case $k=2$ yields the birank of partitions \cite{HL}.
\item[6.] 
We expect that our methods also apply to show  an analogue of \eqref{crankas1} for the rank. 
The case of fixed $m$ is considered in upcoming work by Byungchan Kim, Eunmi Kim, and Jeehyeon Seo \cite{KKS}.
\end{enumerate}
\end{remarks*}

This paper is organized as follows.
In Section 2, we recall basic facts on modular and Jacobi forms which are the base components of $\mathcal{C}_k$ and collect properties on Euler polynomials.
In Section 3, we determine the asymptotic behavior of $\mathcal{C}_k$. 
In Section 4, we use Wright's version of the Circle Method to finish the proof of Theorem \ref{1.4}. In Section $5$, we illustrate Theorem \ref{crank Theorem} numerically.

\section*{Acknowledgements}
The research of the first author was supported by the Alfried Krupp Prize for Young University Teachers of the Krupp foundation and the research leading to these results has received funding from the European Research Council under the European Union's Seventh Framework Programme (FP/2007-2013) / ERC Grant agreement n. 335220 - AQSER. Most of this research was conducted while the second author was visiting the University of Cologne funded by the Krupp foundation. The first author thanks Freeman Dyson and Jan Manschot for enlightening conversation.
Moreover we thank Michael Mertens, Karl Mahlburg, and Larry Rolen for their comments on an earlier version of this paper. Moreover we thank the referee for valuable comments that improved the exposition of this paper.

\section{Preliminaries}
\subsection{Modularity of the generating functions}
A key ingredient of our asymptotic results is to employ the modularity of the functions $\mathcal{C}_k$. To be more precise, we write (throughout $q:=e^{2\pi i \tau}$, $\zeta := e^{2\pi i w}$ with $\tau\in \mathbb{H}, w\in\C$)
\begin{equation}\label{seeJacobi}
\mathcal{C}_k \left( \zeta ; q \right) = \frac{i\left( \zeta^\frac12 - \zeta^{-\frac12}\right) q^{\frac{k}{24}} \eta^{3-k} (\tau)}{\vartheta \left( w; \tau \right)},
\end{equation}
where 
$$
\eta (\tau) := q^{\frac{1}{24}} \prod_{n=1}^\infty \left( 1- q^n \right),
$$
$$
\vartheta \left( w; \tau \right) := i \zeta^\frac12 q^\frac18 \prod_{n=1}^\infty \left( 1- q^n\right) \left( 1 - \zeta q^n \right) \left( 1- \zeta^{-1}q^{n-1}\right).
$$
The function $\eta$ is a modular from, whereas $\vartheta$ is a Jacobi form. To be more precise, we have the following transformation laws (see e.g. \cite{Rad}).

\begin{lemma}\label{LemmaTrans}
We have
$$
\eta \left( - \frac{1}{\tau} \right) = \sqrt{-i\tau} \eta (\tau),
$$
$$
\vartheta \left( \frac{w}{\tau} ; - \frac{1}{\tau} \right) = - i \sqrt{-i \tau} e^{\frac{\pi  i w^2}{\tau}} \vartheta \left( w; \tau \right).
$$
\end{lemma}

\subsection{Euler polynomials}
Recall that the Euler polynomials may be defined by their generating function
\begin{equation}\label{Eulergen}
\frac{2e^{xt}}{e^t +1}=: \sum_{r=0}^\infty E_r (x) \frac{t^r}{r!}.
\end{equation}
The following lemma may easily be concluded by differentiating the generating function \eqref{Eulergen}. For the readers convenience we give a proof.
\begin{lemma}\label{2.2}
We have
$$
-\frac12 \text{\rm sech}^2 \left( \frac{t}2\right) = \sum_{r=0}^\infty E_{2r+1} (0) \frac{t^{2r}}{(2r)!}.
$$
\end{lemma}
\begin{proof}
We have
$$
\sum_{r=0}^\infty E_{2r+1} (0) \frac{t^{2r}}{(2r)!} = \frac{\text{d}}{\text{d}t} \sum_{r=0}^\infty E_{2r+1} (0) \frac{t^{2r+1}}{(2r+1)!}.
$$
Now 
$$
\frac{2}{e^t+1} = \sum_{r=0}^\infty E_r (0) \frac{t^r}{r!},
$$
$$
\frac{2}{e^{-t}+1} = \sum_{r=0}^\infty E_r (0) \frac{(-t)^r}{r!}.
$$
Taking the difference gives the claim of the lemma since
$$
\frac{\text{d}}{\text{d}t} \left( \frac{1}{e^t+1} - \frac{1}{e^{-t}+1}\right) = - \frac12 \text{sech}^2 \left( \frac{t}2\right).
$$
\end{proof}
We also require an integral representation of Euler polynomials. To be more precise, setting for $j\in\N_0$
\begin{equation}\label{Eulerintegral}
\mathcal{E}_j := \int_0^\infty \frac{w^{2j+1}}{\sinh (\pi w)} dw,
\end{equation}
we obtain
\begin{lemma}\label{EulerLemma}
We have
$$
\mathcal{E}_j = \frac{(-1)^{j+1} E_{2j+1} (0)}{2}.
$$
\end{lemma}
\begin{proof}
We make the change of variables $w\to w+\frac{i}{2}$ and then use the Residue Theorem to shift the path of integration back to the real line.
Using the Binomial Theorem, we may thus write
$$
\mathcal{E}_j = -\frac{i}{2} \int_\R \frac{\left(w +\frac{i}2\right)^{2j+1}}{\cosh (\pi w)} dw = -\frac{i}{2} \sum_{\ell=0}^{2j+1}  {2j+1 \choose \ell} \left( \frac{i}2\right)^{2j+1-\ell} \int_\R \frac{w^\ell}{\cosh (\pi w)} dw.
$$
The last integral is known to equal $(-2i)^{-\ell} E_\ell$, where $E_\ell := 2^\ell E_\ell (\frac12)$ denotes the $\ell$th Euler number (see page $41$ of \cite{EMOT}). The claim now follows using the well-known identity
(see page $41$ of \cite{EMOT})
$$
E_j (x) = \sum_{\ell=0}^j {j \choose \ell} \left( x - \frac12 \right)^{j-\ell} \frac{E_\ell}{2^\ell}.
$$
\end{proof}

\section{Asymptotic behavior of the function $\mathcal{C}_{k}$.}
Since $M_k(-m, n)=M_k (m, n)$ we from now on assume that $m\geq 0$.
The goal of this section is to study the asymptotic behavior of the generating function of $M_k(m,n)$.
We define
\begin{align*}
\mathcal{C}_{m,k} (q) &:= \sum_{n=0}^\infty M_k \left( m,n\right) q^n = \int_{-\frac12}^{\frac12} \mathcal{C}_k \left( e^{2\pi i w} ; q \right) e^{-2\pi i mw} dw \\ &= 2 \frac{q^{\frac{k}{24}}}{\eta^k (\tau)} \int_0^{\frac12} g \left( w; \tau \right) \cos (2\pi m w) dw,
\end{align*}
where 
$$
g \left( w; \tau \right) := \frac{i \left( \zeta^{\frac12} - \zeta^{-\frac12} \right) \eta^3 (\tau)}{\vartheta \left( w; \tau \right)}.
$$
Here we used that $g(-w; \tau) = g(w; \tau)$. In this section we determine the asymptotic behavior of $\mathcal{C}_{m,k}(q)$, when $q$ is near an essential singularity on the unit circle. It turns out that the dominant pole lies at $q=1$. Throughout the rest of the paper let $\tau=\frac{iz}{2\pi}, z=\beta_k(1+ixm^{-\frac13})$ with $x\in\R$ satisfying $|x|\leq\frac{\pi m^{\frac13}}{\beta_k}$.
\subsection{Bounds near the dominant pole}
In this section we consider the range $|x|\leq1$. 
We start by determining the asymptotic main term of $g$. Lemma \ref{LemmaTrans} and the definition of $\vartheta$ and $\eta$ immediately imply.
\begin{lemma}\label{boundgLemma}
For $0\leq w \leq 1$ we have for $|x|\leq1$ as $n\rightarrow\infty$
$$
g \left( w; \frac{iz}{2\pi}\right) = \frac{2\pi \sin (\pi w)}{z \sinh \left( \frac{2\pi^2w}{z} \right)}e^{\frac{2\pi^2w^2}{z}} \left( 1 + O \left( e^{-4\pi^2 (1-w) \mathrm{Re}\left(\frac{1}{z}\right) }\right)\right).
$$
\end{lemma}

In view of Lemma \ref{boundgLemma} it is therefore natural to define
\begin{align*}
\mathcal{G}_{m,1} (z) &:= \frac{4\pi}{z} \int_0^{\frac12} \frac{\sin (\pi w)}{\sinh \left( \frac{2\pi^2 w}{z}\right)} e^{\frac{2\pi^2 w^2}{z}} \cos (2\pi m w) dw,\\
\mathcal{G}_{m,2} (z) &:= 2 \int_0^{\frac12} \left( g \left( w; \frac{iz}{2\pi} \right) - \frac{2\pi \sin (\pi w)}{z\sinh \left( \frac{2\pi^2 w}{z}\right)} e^{\frac{2\pi^2 w^2}{z}} \right) \cos (2\pi m w) dw.
\end{align*}
Thus
\begin{equation}\label{Cmk}
\mathcal{C}_{m,k} \left( q \right) = \frac{q^{\frac{k}{24}}}{\eta^k (\tau)} \left( \mathcal{G}_{m,1} (z) + \mathcal{G}_{m,2} (z) \right).
\end{equation}
The dominant contribution comes from $\mathcal{G}_{m, 1}$.
\begin{lemma}\label{MainGm1}
Assume that $|x|\leq 1$ and $m\leq\frac{1}{6\beta_k}\log n$. Then we have as $n\rightarrow \infty$
$$
\mathcal{G}_{m,1} (z) = \frac z4\operatorname{sech}^2\left(\frac{\beta_km}{2}\right)+O\left( \beta_k^2m^{\frac23}\operatorname{sech}^2\left(\frac{\beta_k m}{2}\right)\right).
$$
\end{lemma}
\begin{proof}
Inserting the Taylor expansion of $\sin$, $\exp$, and $\cos$, we get
$$
\sin (\pi w) e^{\frac{2\pi^2 w^2}{z}} \cos (2\pi mw) = \sum_{j,\nu,r\geq 0} \frac{(-1)^{j+\nu}}{(2j+1)! (2\nu)! r!} \pi^{2j+1} (2 \pi m )^{2\nu} \left( \frac{2\pi^2}{z}\right)^r w^{2j+2\nu+2r+1}.
$$
This yields that
$$
\mathcal{G}_{m,1} (z) = \frac{4\pi}{z} \sum_{j,\nu,r\geq 0} \frac{(-1)^{j+\nu}}{(2j+1)! (2\nu)! r!} \pi^{2j+1} (2\pi m)^{2\nu} \left( \frac{2\pi^2}{z}\right)^r \mathcal{I}_{j+\nu+r}
$$
where for $\ell\in\N_0$ we define
$$
\mathcal{I}_\ell := \int_0^{\frac12} \frac{w^{2\ell+1}}{\sinh \left( \frac{2\pi^2w}{z} \right)} dw.
$$
We next relate $\mathcal{I}_\ell$ to $\mathcal{E}_\ell$ defined in (\ref{Eulerintegral}). For this, we note that
\begin{equation}\label{splitI}
\mathcal{I}_\ell = \int_0^{\infty} \frac{w^{2\ell+1}}{\sinh \left( \frac{2\pi^2w}{z} \right)} dw - \mathcal{I}_\ell'
\end{equation}
with 
\begin{multline*}
\mathcal{I}_\ell' := \int_{\frac12}^\infty \frac{w^{2\ell+1}}{\sinh \left( \frac{2\pi^2w}{z} \right)} dw \ll \int_{\frac12}^\infty w^{2\ell+1} e^{-2\pi^2 w \text{Re}\left(\frac1z\right)} dw \\
 \ll \left(\text{Re}\left(\frac{1}{z}\right)\right)^{-2\ell-2} \Gamma \left( 2\ell+2; \pi^2 \text{Re}\left(\frac{1}{z}\right) \right).
\end{multline*}
Here $\Gamma (\alpha ; x) := \int_x^\infty e^{-w}w^{\alpha-1} dw$ denotes the incomplete gamma function and throughout $g(x)\ll f(x)$ means that $g(x)=O\left(f(x)\right)$. Using that as $x\rightarrow \infty$
\begin{equation}\label{incompletebound}
\Gamma \left( \ell; x\right) \sim x^{\ell -1} e^{-x}
\end{equation}
thus yields that
$$
\mathcal{I}_\ell' \ll\left(\text{Re}\left(\frac1{z}\right)\right)^{-1} e^{-\pi^2 \re{\frac{1}{z}}}\leq e^{-\pi^2\text{Re}\left(\frac1{z}\right)}.
$$
In the first summand in (\ref{splitI}) we make the change of variables $w\rightarrow \frac{zw}{2\pi}$ and then shift the path of integration back to the real line by the Residue Theorem. Thus we obtain that
$$
\int_0^\infty \frac{w^{2\ell+1}}{\sinh \left( \frac{2\pi^2w}{z}\right)} dw = \left( \frac{z}{2\pi}\right)^{2\ell+2} \mathcal{E}_\ell  = \left( \frac{z}{2\pi} \right)^{2\ell+2} \frac{(-1)^{\ell+1} E_{2\ell+1} (0)}{2},
$$
where for the last equality we used Lemma \ref{EulerLemma}. Thus 
\begin{multline*}
\mathcal{G}_{m,1} (z) =  \sum_{j,\nu,r\geq 0} \frac{(-1)^{r+1}}{2^{2j+r+1} (2j+1)! (2\nu)! r!} m^{2\nu} z^{2j+2\nu+r+1} \\\times \left( E_{2j+2\nu+2r+1} (0) + O \left( |z|^{-2j-2\nu -2r-2} e^{-\pi^2 \re{\frac{1}{z}}}\right)\right) \\
= \sum_{\nu= 0}^\infty \frac{(mz)^{2\nu}}{(2\nu)!} \left( - \frac{z}{2} E_{2\nu+1} (0) + O \left( |z|^2 \right) \right) = \frac{z}{4} \text{sech}^2 \left( \frac{mz}{2} \right) + O \left( |z|^2 \cosh (mz)\right),
\end{multline*}
where for the last equality we used Lemma \ref{2.2}.
To finish the proof we have to approximate $\operatorname{sech}^2\left(\frac{mz}{2}\right)$ and $\cosh(mz)$. We have 
\begin{align*}
\cosh(mz)&=\cosh\left(\beta_km+i\beta_k m^{\frac23}x\right)\\
&=\cosh(\beta_km)\cos\left(\beta_km^{\frac23}x\right)+i\sinh(\beta_k m)\sin\left(\beta_km^{\frac23}x\right)\\
&=\cosh(\beta_km)\left(1+O\left(\beta_km^{\frac23}\right)\right).
\end{align*}

This implies that 
\[\operatorname{sech}\left(\frac{mz}{2}\right)=\frac{1}{\cosh\left(\frac{mz}{2}\right)}=\frac{1}{\cosh\left(\frac{\beta_km}{2}\right)\left(1+O\left(\beta_km^{\frac23}\right)\right)},\]

yielding 

\[\operatorname{sech}^2\left(\frac{mz}{2}\right)=\operatorname{sech}^2\left(\frac{\beta_km}{2}\right)\left(1+O\left(\beta_km^{\frac23}\right)\right).\]

Thus we obtain

\[\mathcal{G}_{m,1}(z)=\frac z4\operatorname{sech}^2\left(\frac{\beta_km}{2}\right)\left(1+O\left(\beta_km^{\frac23}\right)\right)+O\left(\beta_k^2\left(1+\frac{1}{m^{\frac23}}\right)\cosh(\beta_km)\right)\]

\[=\frac z4\operatorname{sech}^2\left(\frac{\beta_km}{2}\right)+O\left(\beta_k^2m^{\frac23}\operatorname{sech}^2\left(\frac{\beta_km}{2}\right)\right)+O\left(\beta_k^2\cosh(\beta_km)\right).\]

We may now easily finish the proof distinguishing the cases on whether $\beta_k m$ is bounded or goes to $\infty$.

\end{proof}
We next turn to bounding $\mathcal{G}_{m,2}$.
\begin{lemma}\label{boundG2Lemma}
Assume that $|x|\leq 1$. Then we have as $n\rightarrow\infty$
$$
\mathcal{G}_{m,2} (q) \ll \frac{1}{\beta_k} e^{-\frac{5\pi^2}{4\beta_k}}.$$
\end{lemma}
\begin{proof}
By Lemma \ref{boundgLemma} we obtain that
$$
\mathcal{G}_{m,2} (z) \ll \frac{1}{|z|} \int_0^{\frac12} \left| \frac{\sin(\pi w)}{1- e^{-\frac{4\pi^2 w}{z}}}\right| e^{2\pi^2 \re{\frac{1}{z}} \left( w^2+w-2\right)} dw.
$$
It is not hard to see that
$$
\left| \frac{\sin(\pi w)}{1- e^{-\frac{4\pi^2 w}{z}}}\right| \ll 1.
$$

Moreover,
\begin{align*}
|z|&=\beta_k\sqrt{1+m^{-\frac23}x^2}\gg\beta_k,\\
\text{Re}\left(\frac1{z}\right)&\geq \frac1{2\beta_k}.
\end{align*}
The claim now follows, using that the maximum of $w^2+w-2$ on $[0,\frac12]$ is obtained for $w=\frac12$.
\end{proof}
Combining the above yields.
\begin{proposition}\label{MainCmProp}
Assume that $|x|\leq 1$. Then we have as $n\rightarrow \infty$
$$
\mathcal{C}_{m,k} \left( q \right) = \frac{z^{\frac k2+1}}{4(2\pi)^\frac k2}\operatorname{sech}^2\left(\frac{\beta_k m}{2}\right)e^{\frac{k\pi ^2}{6z}}+O\left(\beta_k^{\frac{k}{2}+2}m^{\frac23}\operatorname{sech}^2\left(\frac{\beta_km}{2}\right)e^{\pi\sqrt{\frac{kn}{6}}}\right).
$$
\end{proposition}
\begin{proof}
Recall from (\ref{Cmk}) that
$$
\mathcal{C}_{m,k}(q)=\frac{q^\frac {k}{24}}{\eta^k(\tau)}\left(\mathcal{G}_{m,1}(z)+\mathcal{G}_{m,2}(z)\right).
$$
Lemma \ref{LemmaTrans} easily gives that
$$
\frac{q^\frac{k}{24}}{\eta^k(\tau)}=\left(\frac{z}{2\pi}\right)^\frac k2 e^{\frac{k\pi^2}{6z}}\left(1+O(\beta_k)\right).
$$
The functions $\mathcal{G}_{m, 1}$ and $\mathcal{G}_{m, 2}$ are now approximated using Lemma \ref{MainGm1} and Lemma \ref{boundG2Lemma}, respectively.
It is not hard to see that the main error term arises from approximation $\mathcal{G}_{m, 1}$. We thus obtain
\[
\mathcal{C}_{m, k}(q)= \frac{z^{\frac{k}{2}+1}}{4(2\pi)^{\frac{k}{2}}}e^{\frac{k\pi^2}{6z}}\operatorname{sech}^2\left(\frac{\beta_k m}{2}\right)+O\left(|z|^{\frac{k}{2}}\beta_k^2 m^{\frac23}\operatorname{sech}^2\left(\frac{\beta_k m}{2}\right)e^{\frac{\pi^2k}{6}\text{Re}\left(\frac1{z}\right)}\right).
\]
The claim follows now using that
\begin{align*}
|z|&\ll \beta_k,\\
\text{Re}\left(\frac1{z}\right)&\leq \frac1{\beta_k}=\frac{\sqrt{6n}}{\pi\sqrt{k}}.
\end{align*}
\end{proof}

\subsection{Bounds away from the dominant pole}

We next investigate the behavior of $\mathcal{C}_{m,k}$ away from the dominant cusp $q=1$. To be more precise, we consider the range $1\leq x\leq\frac{\pi m^{\frac13}}{\beta_k}$.
Let us start with the following lemma, which proof uses the same idea as in ~\cite{W1}.

\begin{lemma}\label{lem:far}
Assume that $\tau=u+iv\in \mathbb{H}$ with $Mv \leq |u| \leq \frac{1}{2}$ for $u>0$ and $v\rightarrow 0$, we have that
\[ \left|P(q)\right| \ll \sqrt{v} \exp \left[\frac{1}{v}\left(\frac{\pi}{12}-\frac{1}{2\pi}\left(1-\frac{1}{\sqrt{1+M^2}}\right)\right)\right].\]
\end{lemma}
\begin{proof}
We rewrite
\begin{align*}
\log(P(q))&=-\sum\limits_{n=1}^\infty \log(1-q^n)=\sum\limits_{n=1}^\infty\sum\limits_{m=1}^\infty \frac{q^{nm}}{m}                 =\sum\limits_{m=1}^\infty \frac{q^m}{m(1-q^m)}.
\end{align*}
Therefore we may estimate
\begin{align*}
|\log(P(q))|&\leq\sum\limits_{m=1}^\infty \frac{|q|^m}{m|1-q^m|} \leq \frac{|q|}{|1-q|}-\frac{|q|}{1-|q|}+\sum\limits_{m=1}^\infty \frac{|q|^m}{m(1-|q|^m)}\\
                       &=\log(P(|q|))-|q|\left(\frac{1}{1-|q|}-\frac{1}{|1-q|}\right).
\end{align*}
We now split $u$ into $2$ ranges.
If $Mv \leq |u| \leq \frac{1}{4},$ then we have $\cos(2\pi u) \leq \cos(2 \pi Mv)$. Therefore
\[|1-q|^2=1-2e^{-2\pi v}\cos(2\pi u)+e^{-4\pi v}\geq 1-2e^{-2\pi v}\cos(2\pi Mv)+e^{-4\pi v}.\]
Taylor expanding around $v=0$ we find that
\begin{equation}
\label{eq:closefar}
|1-q|\geq 2\pi v \sqrt{1+M^2}+O\left(v^2\right).
\end{equation}
For $\frac{1}{4}\leq |u| \leq \frac{1}{2}$ we have $\cos(2\pi u) \leq 0$. Therefore
$$ |1-q| \geq 1> 2\pi v \sqrt{1+M^2}.$$ 
Hence, for all $Mv \leq |u| \leq \frac{1}{2}$,
\begin{equation}
\label{eq:|1-q|}
|1-q|\geq 2\pi v \sqrt{1+M^2}+O\left(v^2\right).
\end{equation}
Furthermore we have
\begin{equation}
\label{eq:1-q}
1 - |q| = 1-e^{-2\pi v} = 2\pi v +O\left(v^2\right).
\end{equation}
By Lemma~\ref{LemmaTrans}, we have
\begin{align*}
P(|q|)=\frac{e^{-\frac{2\pi v}{24}}}{\eta(iv)} = \sqrt{v} e^{\frac{\pi}{12v}}\left(1+O(v)\right).
\end{align*}
Thus
\begin{equation}
\label{eq:logP}
\log(P(|q|)) = \frac{\pi}{12v}+\frac{1}{2}\log(v) + O(v).
\end{equation}
Combining~\eqref{eq:|1-q|},~\eqref{eq:1-q}, and~\eqref{eq:logP}, we obtain
\begin{align*}
|\log(P(q))|&\leq \frac{\pi}{12v}+\frac{1}{2}\log(v) + O(v)-\frac{1}{2\pi v}\left(1-\frac{1}{\sqrt{1+M^2}}\right)+O(1)\\
                       &=\frac{1}{v}\left(\frac{\pi}{12}-\frac{1}{2\pi}\left(1-\frac{1}{\sqrt{1+M^2}}\right)\right)+\frac{1}{2}\log(v)+O(1).
\end{align*}
Exponentiating yields the desired result.
\end{proof}

We are now able to bound $|\mathcal{C}_{m,k}(q)|$ away from $q=1$.

\begin{proposition}\label{far}
Assume that $1\leq|x|\leq \frac{\pi m^\frac 13}{\beta}$. Then we have, as $n\rightarrow \infty$,
$$
|\mathcal{C}_{m,k}(q)| \ll n^{\frac{3-k}{4}} \exp \left(\pi \sqrt{\frac{kn}{6}} - \frac{\sqrt{6kn}}{8\pi} m^{-\frac{2}{3}}\right).
$$
\end{proposition}
\begin{proof}
We have
by definition
$$
\mathcal{C}_{m,k}(q) = 2 P^k(q)  \int_{0}^{\frac{1}{2}} g(w;\tau) \cos(2\pi mw) dw.
$$
Note that by (\ref{CrankLerch})
$$
g\left( w; \tau\right) = 
1+ \left(1-\zeta\right) \sum_{n \geq 1} \frac{(-1)^nq^{\frac{n^2+n}{2}}}{1-\zeta q^n} + \left(1-\zeta^{-1}\right) \sum_{n \geq 1} \frac{(-1)^nq^{\frac{n^2+n}{2}}}{1-\zeta^{-1} q^n}.
$$

We thus may bound 
$$
g\left( w; \tau \right) \ll
\sum_{n\geq 1} \frac{|q|^{\frac{n^2+n}2}}{1-|q|^n} \ll
\frac{1}{1-|q|} \sum_{n\geq 1} e^{-\frac{\beta_k n^2}2} \ll
\beta_k^{-\frac32} \ll n^{\frac34}.
$$
Thus
$$
\left| \mathcal{C}_{m,k} (q) \right| \ll
\left| P^k (q)\right| n^{\frac34}.
$$

Using Lemma~\ref{lem:far} with $v=\frac{\beta_k}{2\pi}$, $u=\frac{\beta_k m^{-\frac{1}{3}}x}{2\pi}$, and $M = m^{-\frac{1}{3}}$, yields for $1\leq|x|\leq \frac{\pi m^\frac 13}{\beta_k}$,
\begin{align*}
|P(q)| &\ll n^{-\frac{1}{4}} \exp \left[\frac{2\pi}{\beta_k}\left(\frac{\pi}{12}-\frac{1}{2\pi}\left(1-\frac{1}{\sqrt{1+m^{-\frac{2}{3}}}}\right)\right)\right].
\end{align*}
Therefore
\begin{align*}
|\mathcal{C}_{m,k}(q)| &\ll n^{\frac{3-k}{4}} \exp \left[\frac{2\pi k}{\beta_k}\left(\frac{\pi}{12}-\frac{1}{2\pi}\left(1-\frac{1}{\sqrt{1+m^{-\frac{2}{3}}}}\right)\right)\right]\\
&\ll n^{\frac{3-k}{4}} \exp \left[\pi \sqrt{\frac{kn}{6}} - \frac{\sqrt{6kn}}{\pi} \left(1-\frac{1}{\sqrt{1+m^{-\frac{2}{3}}}}\right)\right]\\
&\ll n^{\frac{3-k}{4}} \exp \left(\pi \sqrt{\frac{kn}{6}} - \frac{\sqrt{6kn}}{8\pi} m^{-\frac{2}{3}}\right).
\end{align*}

\end{proof}

\section{The Circle Method}
In this section we use Wright's variant of the Circle Method and complete the proof of Theorem \ref{1.4} and thus the proof of Dyson's conjecture. We start by using Cauchy's Theorem to express $M_k$ as an integral of its generating function $\mathcal{C}_{m ,k}$: 
\begin{equation}\label{intrep}
M_k \left( m,n \right) = \frac{1}{2\pi i } \int_C \frac{\mathcal{C}_{m,k} (q)}{q^{n+1}} dq ,
\end{equation}
where the contour is the counterclockwise transversal of the circle $C := \{q\in\C\: ;\: |q| = e^{-\beta_k}\}$. Recall that $z=\beta_k(1+ixm^{-\frac 13})$. Changing variables we may write
$$
M_k(m,n)=\frac{\beta_k}{2\pi m^\frac 13}\int\limits_{|x|\leq \frac{\pi m^\frac 13}{\beta_k}} \mathcal{C}_{m,k}(e^{-z})e^{nz} dx.
$$
We split this integral into two pieces
$$
M_k (m,n) = M + E
$$
with
$$
M:=\frac{\beta_k}{2\pi m^\frac 13} \int_{|x|\leq 1} \mathcal{C}_{m,k}  \left(e^{-z} \right)e^{nz} dx,
$$
$$
E:=\frac{\beta_k}{2\pi m^\frac 13} \int_{1\leq |x|\leq \frac{\pi m^{\frac13}}{\beta_k}} \mathcal{C}_{m,k}  \left(e^{-z} \right)e^{nz} dx.
$$
In the following we show that $M$ contributes to the asymptotic main term whereas $E$ is part of the error term.
\subsection{Approximating the main term}
The goal of this section is to determine the asymptotic behavior of $M$. We show
\begin{proposition}\label{boundMLemma}
We have
$$
M =\frac{\beta_k}{4}\operatorname{sech}^2\left(\frac{\beta_k m}{2}\right)p_k(n)\left(1+O\left( \frac{m^\frac 13}{n^\frac 14} \right)\right).
$$
\end{proposition}
A key step for proving this proposition is the investigation of
$$
P_{s,k} := \frac{1}{2\pi i}\int\limits_{1-im^{-\frac 13}}^{1+im^{-\frac 13}} v^s e^{\pi\sqrt{\frac{kn}{6}}\left(v+\frac 1v\right)}dv
$$
for $s>0$. These integrals may be related to Bessel functions. Denoting by $I_s$ the usual $I$-Bessel function of order $s$, we have.
\begin{lemma}\label{4.2}
As $n\rightarrow\infty$ 
$$
P_{s,k} = I_{-s-1} \left(\pi\sqrt{\frac{2kn}{3}} \right) +O\left( \exp\left(\pi\sqrt{\frac{kn}{6}}\left(1+\frac{1}{1+m^{-\frac 23}}\right)\right)\right).
$$
\end{lemma}
\begin{proof}
We use the following loop integral representation for the $I$-Bessel function \cite{AS} ($x>0$) 
\begin{equation}\label{Ar}
I_\ell (2x) = \frac{1}{2\pi i} \int_\Gamma t^{-\ell-1} e^{x \left( t + \frac{1}{t} \right)} dt,
\end{equation}
where the contour $\Gamma$ starts in the lower half plane at $-\infty$, surrounds the origin counterclockwise and then returns to $-\infty$ in the upper half-plane. We choose for $\Gamma$ the piecewise linear path that consists of the line segments
\begin{align*}
&\gamma_4:\,\left(-\infty-\frac{i}{2m^\frac 13},-1-\frac{i}{2m^\frac 13}\right),\quad  \gamma_3:\,\left(-1-\frac{i}{2m^\frac 13},-1-\frac{i}{m^\frac 13}\right),\\
&\gamma_2:\,\left(-1-\frac{i}{m^\frac 13},1-\frac{i}{m^\frac 13}\right),\quad \gamma_1:\left(1-\frac{i}{m^\frac 13},1+\frac{i}{m^\frac 13}\right),\end{align*}
which are then followed by the corresponding mirror images $\gamma_2',\gamma_3',$ and $\gamma_4'$. Note that $P_{s,k}=\int_{\gamma_1}$. Thus, to finish the proof, we have to bound the integrals along $\gamma_4,\gamma_3,$ and $\gamma_2 -$  the corresponding mirror images follow in the same way.

First
\begin{align*}
\int\limits_{\gamma_4}
&\ll \int_{-\infty}^1\left|\exp\left(\pi \sqrt{\frac{kn}{6}}\left(t-\frac{im^{-\frac13}}{2}+\frac1{t-\frac{im^{-\frac13}}{2}}\right)\right)\right|\left|t-\frac{im^{-\frac13}}{2}\right|^s dt\\
&\ll \int\limits_{1}^{\infty} e^{-\pi\sqrt{\frac{kn}{6}}t}\left| t+\frac{im^{-\frac 13}}{2}\right|^s dt\\
&\ll \int_1^\infty t^s e^{-\pi \sqrt{\frac{kn}{6}}t}dt\ll n^{-\frac{s+1}{2}}\Gamma\left(s+1; \pi\sqrt{\frac{kn}{6}}\right)\ll n^{-\frac 12} e^{-\pi\sqrt\frac{kn }{6}},
\end{align*}
using \eqref{incompletebound}.

Next
$$
\int\limits_{\gamma_3}\ll m^{-\frac 13}\int\limits_{\frac 12}^1 \exp\left(-\pi\sqrt{\frac{kn}{6}}\left(1+\frac1{1+m^{-\frac23}t^2}\right)\right)\left|1+im^{-\frac 13} t\right|^s dt\ll e^{-\pi\sqrt{\frac{kn}{6}}}.
$$
Finally
\begin{align*}
&\int\limits_{\gamma_2} \ll\int\limits_{-1}^1 \exp\left(\pi\sqrt{\frac{kn}{6}}\left(t+\frac{t}{t^2+m^{-\frac23}}\right)\right)\left|t-im^{-\frac 13}\right|^s dt\\
&\ll \exp\left(\pi\sqrt{\frac{kn}{6}}\left(1+\frac{1}{1+m^{-\frac 23}}\right)\right),
\end{align*}
where we used that $t+\frac{t}{t^2 rm^{-\frac23}}$ obtains its maximum at $t=1$.
This finishes the proof.
\end{proof}
We now turn to the proof of Proposition \ref{boundMLemma}. 

\begin{proof}[Proof of Proposition \ref{boundMLemma}]
Using Proposition \ref{MainCmProp} and making a change of variables, we obtain by Lemma \ref{4.2}
\begin{multline*}
 M= \frac{\beta_k^{\frac{k}{2}+2}}{4(2\pi)^{\frac{k}{2}}}\operatorname{sech}^2\left(\frac{\beta_k m}{2}\right)P_{\frac{k}{2}+1,k} + O\left(\beta_k^{\frac{k}{2}+3}m^{\frac{1}{3}}\operatorname{sech}^2\left(\frac{\beta_k m}{2}\right)  e^{\pi \sqrt{\frac{2kn}{3}}}\right) \\
     =      \frac{\beta_k^{\frac{k}{2}+2}}{4(2\pi)^{\frac{k}{2}}}\operatorname{sech}^2\left(\frac{\beta_k m}{2}\right)\Bigg(I_{-\frac{k}{2}-2}\left(\pi \sqrt{\frac{2kn}{3}}\right)\\
      + O\left(\exp\left(\pi \sqrt{\frac{kn}{6}}\left(1 + \frac{1}{1 + m^{-\frac{2}{3}}}\right)\right)\right) \Bigg)
  + O\left(\beta_k^{\frac{k}{2}+3}m^{\frac{1}{3}}\operatorname{sech}^2\left(\frac{\beta_k m}{2}\right)e^{\pi\sqrt{\frac{2kn}{3}}}\right).
\end{multline*}
Using the Bessel function asymptotic (see (4.12.7) in \cite{AAR})
$$
I_{\ell}(x)= \frac{e^x}{\sqrt{2 \pi x}} + O\left(\frac{e^x}{x^{\frac{3}{2}}}\right)
$$
yields
\begin{multline*}
 M = \frac{\beta_k^{\frac{k}{2}+2}}{4(2\pi)^{\frac{k}{2}}}\operatorname{sech}^2\left(\frac{\beta_k m}{2}\right)\left( \frac{ e^{\pi \sqrt{\frac{2kn}{3}}}}{\pi \sqrt{2}\left(\frac{2kn}{3}\right)^\frac{1}{4}} + O\left(\frac{ e^{\pi \sqrt{\frac{2kn}{3}}}}{n^\frac{3}{4}} \right) \right. \\ 
+ \left. O\left(\exp\left(\pi \sqrt{\frac{kn}{6}}\left(1 + \frac{1}{1 + m^{-\frac{2}{3}}}\right)\right)\right) \right) + O\left(\beta_k^{\frac{k}{2}+3}m^{\frac{1}{3}}\operatorname{sech}^2\left(\frac{\beta_k m}{2}\right)e^{\pi \sqrt{\frac{2kn}{3}}}\right).
\end{multline*}
It is not hard to see that the last error term is the dominant one.
Thus
\[
 M= \frac{\beta_k^{\frac{k}{2}+2}}{4(2\pi)^{\frac{k}{2}}}\operatorname{sech}^2\left(\frac{\beta_k m}{2}\right)\frac{e^{\pi \sqrt{\frac{2 k n}{3}}}}{\pi \sqrt{2}\left(\frac{2 k n}{3}\right)^{\frac{1}{4}}}
 \left(1+O\left(m^{\frac13}n^{-\frac14}\right)\right).
\]
Using that \cite{HR, RZ}
$$
p_k(n) = 2\left(\frac{k}{3}\right)^{\frac{1+k}{4}}(8n)^{-\frac{3+k}{4}} e^{\pi \sqrt{\frac{2kn}{3}}}\left(1+O\left(\frac1{\sqrt{n}}\right)\right)
$$
now easily gives the claim. 
\end{proof}

\subsection{The error arc.} 

We finally bound $E$ and show that it is exponentially smaller than $M$.
The following proposition then immediately implies Theorem \ref{1.4}.

\begin{proposition}
\label{integralE}
As $n\rightarrow \infty$
$$
E \ll n^{\frac{3-k}{4}} \exp \left(\pi \sqrt{\frac{2kn}{3}} - \frac{\sqrt{6kn}}{8\pi} m^{-\frac{2}{3}}\right).
$$
\end{proposition}

\begin{proof}
Using Proposition~\ref{far}, we may bound
\begin{align*}
E&\ll \frac{\beta_k}{m^{\frac13}}\int_{1\leq x\leq\frac{\pi m^{\frac13}}{\beta_k}}  n^{\frac{3-k}{4}} \exp \left(\pi \sqrt{\frac{kn}{6}} - \frac{\sqrt{6kn}}{8\pi} m^{-\frac{2}{3}}\right) e^{\beta_k n}  dx\\
&\ll n^{\frac{3-k}{4}} \exp \left(\pi \sqrt{\frac{2kn}{3}} - \frac{\sqrt{6kn}}{8\pi} m^{-\frac{2}{3}}\right).
\end{align*}
\end{proof}

\section{Numerical data}
We illustrate our results in $2$ tables.

\begin{center}
  \begin{tabular}{ c | c | c | c  }
    $n$ & $M\left(0,n\right)$ & $\widetilde{M}\left(0,n\right)$ & $\frac{M\left(0,n\right)}{\widetilde{M}\left(0,n\right)}$  \\ \hline
   $20$ & $41$ & $\sim 45$ & $\sim 0.912$ \\ 
   $50$ & $8626$ & $\sim 9261$ & $\sim 0.931$ \\ 
   $500$ & $3.228743492 \cdot 10^{19}$ & $\sim 3.298285542\cdot 10^{19}$ & $\sim 0.979$ \\ 
   $1000$ & $2.403603986\cdot 10^{29}$ & $\sim 2.439699707 \cdot 10^{29}$ & $\sim 0.985$ \\ 
  \end{tabular}
\end{center}

\begin{center}
  \begin{tabular}{ c | c | c | c  }
    $n$ & $M\left(1,n\right)$ & $\widetilde{M}\left(1,n\right)$ & $\frac{M\left(1,n\right)}{\widetilde{M}\left(1,n\right)}$  \\ \hline
   $20$ & $38$ & $\sim 44$ & $\sim 0.863$ \\ 
   $50$ & $8541$ & $\sim 9185$ & $\sim 0.930$ \\ 
   $500$ & $3.226300403 \cdot 10^{19}$ & $\sim 3.295574297\cdot 10^{19}$ & $\sim 0.979$ \\ 
   $1000$ & $2.402671309\cdot 10^{29}$ & $\sim 2.438696696 \cdot 10^{29}$ & $\sim 0.985$ \\ 
  \end{tabular}
\end{center}

where we set $\widetilde{M}(m,n) := \frac{\beta}4 \text{sech}^2 (\frac{\beta m}{2}) p(n)$.


\begin{thebibliography}{99}

\bibitem{AG} G.  Andrews and F. Garvan,  {\it Dyson's crank of a partition},
Bull. Amer. Math. Soc. {\bf 18} (1988), 167--171.

\bibitem{AR} A. Atkin and H.  Swinnerton-Dyer,
\emph{Some properties of partitions}, Proc. Lond. Math. Soc.
\textbf{4} (1954), 84--106.

\bibitem{AAR} G. Andrews, R. Askey, and R.Roy, \emph{Special functions},
Encyclopedia Math. Appl. {\bf 71}, 
Cambridge University Press, Cambridge 1999.

\bibitem{AS} G. Arken,
\emph{Modified Bessel functions},  Mathematical Methods for Physicists, 3rd ed., Orlando, FL: Academic Press (1985), 610--616.

\bibitem{BMR} K. Bringmann, K. Mahlburg, and R. Rhoades, \emph{Asymptotics for rank and crank moments}, Bull. Lond. Math. Soc. {\bf 43} (2011), 661-672.

\bibitem{BM} K. Bringmann and J. Manschot, \emph{Asymptotic formulas for coefficients of inverse theta functions}, submitted for publication.

\bibitem{BO1} K. Bringmann and K. Ono, \emph{Dyson's ranks and Maass forms},
Ann. of Math. {\bf 171} (2010), 419--449.


\bibitem{Dy1} F. Dyson, \emph{Some guesses in the theory of partitions}, Eureka (Cambridge) \textbf{8} (1944), 10--15.

\bibitem{Dy2} F. Dyson, \emph{Mappings and symmetries of partitions}, J. Combin. Theory, Ser. A \textbf{51} (1989), 169--180.

\bibitem{EMOT} A. Erd\'elyi, W. Magnus, F. Oberhettinger, and F. Tricomi, \emph{Higher transcendental functions}, Vol. \rom{1}, Robert E. Krieger Publishing Co., Inc., Melbourne, Fla. (1981)

\bibitem{FM} A. Folsom and R. Masri, \emph{Equidistribution of Heegner points and the partition function}, Math. Ann. {\bf 348} (2011), 3724--3759.

\bibitem{Ga} F. Garvan, \emph{New combinatorial interpretations of Ramanujan's partition congruences mod $5$, $7$, and $11$}, Trans. Amer. Math. Soc. {\bf 305} (1988), 47--77.


\bibitem{Go}
L. G\"ottsche, {\it The Betti numbers of the Hilbert scheme of points
  on a smooth projective surface}, Math.\ Ann. {\bf 286} (1990), 193--207.

\bibitem{HL} P. Hammond and R. Lewis, \emph{Congruences in ordered pairs of partitions}, Int. J. Math. {\bf 2004} (2004), 2509-2512.

\bibitem{HR} G. Hardy and S. Ramanujan, \emph{Asymptotic formulae for the distribution of integers of various types}, Proc. Lond. Math. Soc. {\bf 16} (1918), 112--132.


\bibitem{HRo} T. Hausel and Rodriguez-Villegas \emph{Cohomology of large semiprojective hyperkaehler varieties}, preprint. 

\bibitem{KKS} B. Kim, E. Kim, and J. Seo, \emph{Asymptotics for $q$-expansions involving partial theta functions}, in preparation.

\bibitem{Le} D. Lehmer, \emph{On the remainders and convergence of the series for the partition function}, Trans. Amer. Math. Soc. {\bf 46} (1939), 362--373.

\bibitem{Ma} K. Mahlburg, \emph{Partition congruences and the Andrews-Garvan-Dyson crank}, Proc. Natl. Acad. Sci. {\bf 102} (2005), 15373--15376.

\bibitem{Rad} H. Rademacher, \emph{Topics in analytic number theory}, Die Grundlagen der math. Wiss., Band {\bf 169}, Springer Verlag, Berlin (1973).

\bibitem{RZ} H. Rademacher and H. Zuckerman, \emph{On the Fourier coeffcients of certain modular forms of positive dimension}, Ann. of Math. {\bf 39} (1938), 433--462.

\bibitem{Ra} S. Ramanujan, \emph{Congruence properties of partitions}, Math. Z. \textbf{9} (1921), 147--153.

\bibitem{W1} E. Wright, \emph{Asymptotic partition formulae \rom{2}. Weighted partitions.}, Proc. Lond. Math. Soc. \textbf{36} (1933), 117--141.

\bibitem{W2} E. Wright, \emph{Stacks (\rom{2})}, Q. J. Math. \textbf{22} (1971), 107--116.

\end{thebibliography}
\end{document}